\documentclass[12pt,a4paper,reqno]{amsart}
\usepackage{amssymb}
\usepackage{amscd}
\usepackage{enumerate}
\usepackage{graphicx}
\usepackage{siunitx}
\usepackage{tikz-cd}
\usepackage{color}
\usetikzlibrary{arrows}
\numberwithin{equation}{section}

\usepackage{mathabx}

\usepackage{mathtools}
\usepackage[tableposition=top]{caption}
\usepackage{booktabs,dcolumn}

\DeclareFontFamily{OT1}{rsfs}{}
\DeclareFontShape{OT1}{rsfs}{n}{it}{<-> rsfs10}{}
\DeclareMathAlphabet{\mathscr}{OT1}{rsfs}{n}{it}

\theoremstyle{plain}

\newtheorem{theorem}{Theorem}[section]
\newtheorem{proposition}[theorem]{Proposition}
\newtheorem{lemma}[theorem]{Lemma}
\newtheorem{corollary}[theorem]{Corollary}

\theoremstyle{definition}

\newtheorem{definition}[theorem]{Definition}
\newtheorem{remark}[theorem]{Remark}

\newcommand\R{\mathbb{R}}
\newcommand\Z{\mathbb{Z}}

\newcommand\eps{\varepsilon}

\begin{document}

\title[Universality of potential well dynamics]{On the universality of potential well dynamics}

\author{Terence Tao}
\address{UCLA Department of Mathematics, Los Angeles, CA 90095-1555.}
\email{tao@math.ucla.edu}



\begin{abstract}  Given a smooth potential function $V \colon \R^m \to \R$, one can consider the ODE $\partial_t^2 u = -(\nabla V)(u)$ describing the trajectory of a particle $t \mapsto u(t)$ in the potential well $V$.  We consider the question of whether the dynamics of this family of ODE are \emph{universal} in the sense that they contain (as embedded copies) any first-order ODE $\partial_t u = X(u)$ arising from a smooth vector field $X$ on a manifold $M$.  Assuming that $X$ is nonsingular and $M$ is compact, we show (using the Nash embedding theorem) that this is possible precisely when the flow $(M,X)$ supports a geometric structure which we call a \emph{strongly adapted $1$-form}; many smooth flows do have such a $1$-form, but we give an example (due to Bryant) of a flow which does not, and hence cannot be modeled by the dynamics of a potential well.  As one consequence of this embeddability criterion, we construct an example of a (coercive) potential well system which is \emph{Turing complete} in the sense that the halting of any Turing machine with a given input is equivalent to a certain bounded trajectory in this system entering a certain open set.  In particular, this system contains trajectories for which it is undecidable whether that trajectory enters such a set.

Remarkably, the above results also hold if one works instead with the nonlinear wave equation $\partial_t^2 u - \Delta u = -(\nabla V)(u)$ on a torus instead of a particle in a potential well, or if one replaces the target domain $\R^m$ by a more general Riemannian manifold.
\end{abstract}

\maketitle


\section{Introduction}

Define a \emph{smooth flow} to be a pair $(M,X)$ consisting of a smooth manifold $M$ and a vector field\footnote{In this paper, all vector fields, differential forms, Riemannian metrics, Hamiltonians, potential functions, etc. are understood to be smooth.} $X$ on $M$.  Define a \emph{trajectory} of a smooth flow to be a solution $u \colon I \to M$ to the first-order ordinary differential equation (ODE)
\begin{equation}\label{uxt}
\partial_t u = X(u)
\end{equation}
for some interval $I \subset \R$.  The Picard existence and uniqueness theorem asserts that for any initial datum $u_0 \in M$, there is a unique trajectory $u \colon I \to M$ to \eqref{uxt} with initial data $u(0) = u_0$ and with a maximal open interval of existence $0 \in I \subset \R$; furthermore, under reasonable growth conditions on $X$ (e.g. if $X$ is bounded) the solution is global in the sense that $I=\R$.  In particular, when $M$ is compact all trajectories can be extended to be global in time, and we can define flow maps $e^{tX} \colon M \to M$ for any time $t$; in this case the dynamics are \emph{almost periodic} since all trajectories are clearly precompact.  We say that a smooth flow is \emph{nonsingular} if there are no fixed points, or equivalently if the vector field $X$ is nowhere vanishing.

Define a \emph{morphism}  of one smooth flow $(M,X)$ to another $(M',X')$ to be a smooth map $\phi \colon M \to M'$ that takes trajectories of $(M,X)$ to trajectories of $(M',X')$, or equivalently that
$$ d\phi(X(y)) = X'(\phi(y))$$
for all $y \in M$.  Define an \emph{embedding} of $(M,X)$ into $(M',X')$ is a morphism $\phi \colon M \to M'$ which is also an injective immersion.

Informally, the presence of an embedding of $(M,X)$ into $(M',X')$ indicates that the dynamics of the former system are contained in that of the latter.  For instance:
\begin{itemize}
\item A stationary solution in $(M,X)$ is the same thing as an embedding into $(M,X)$ of the trivial flow $(\mathrm{pt}, 0)$;
\item A periodic solution in $(M,X)$ is the same thing as an embedding into $(M,X)$ of the circle shift $(\R/\Z, 1)$;
\item An invariant torus in $(M,X)$ (in the sense of KAM theory) is the same thing as an embedding of $(M,X)$ of a torus shift $((\R/\Z)^d, \alpha)$ for some constant velocity field $\alpha \in \R^d$.
\end{itemize}

Let us say that a class ${\mathcal C}$ of smooth flows is \emph{universal} if any other smooth flow $(M,X)$ may be embedded in at least one system in this class ${\mathcal C}$.  Here is a simple example of such a universal class:

\begin{proposition}[Hamiltonian dynamics are universal]  Let ${\mathcal H}$ be the class of \emph{Hamiltonian} flows $(M,X)$, that is to say smooth flows in which $M = (M,\omega)$ is a symplectic manifold, and there is a Hamiltonian $H \colon M \to \R$ with the property that $\omega(X, Y) = {\mathcal L}_Y H$ for all vector fields $Y$, where we use ${\mathcal L}_Y$ to denote the Lie derivative (which in this case is the same as the ordinary derivative since $H$ is scalar).  Then ${\mathcal H}$ is universal.
\end{proposition}

\begin{proof}  Let $(M,X)$ be a smooth flow.  As is well known, the cotangent bundle $T^* M$ of $M$ can be equipped with a symplectic form $\omega$, which in local coordinates is given by
\begin{equation}\label{omega}
 \omega = \sum_i dq_i \wedge dp_i
\end{equation}
where $q_1,\dots,q_n$ are local coordinates of $M$, and $p_1,\dots,p_n$ are the dual momentum coordinates.  The ODE associated to a Hamiltonian $H \colon T^* M \to \R$ is given in coordinates by Hamilton's equation of motion
\begin{equation}\label{qp}
 \partial_t q_i = \frac{\partial H}{\partial p_i}; \quad \partial_t p_i = - \frac{\partial H}{\partial q_i}.
\end{equation}
If one chooses the specific Hamiltonian $H \colon T^* M \to \R$ defined by the formula
$$ H( q, p ) \coloneqq p(X)$$
for any point $q$ in $M$ and any covector $p \in T^*_q M$, or in coordinates
\begin{equation}\label{hamil}
 H( q_1,\dots, q_n, p_1, \dots, p_n) = \sum_i p_i X_i(q_1,\dots,q_n),
\end{equation}
then one easily checks that the map $\phi  \colon M \to T^* M$ given by $\phi(q) \coloneqq (q,0)$ is an embedding; in coordinates, this asserts that any solution $q  \colon I \to M$ to the ODE $\partial_t q = X(q)$ can also be viewed as solutions to Hamilton's equations of motion \eqref{qp} for the Hamiltonian \eqref{hamil} by setting $p(t) = 0$ for all times $t \in I$.
\end{proof}

Informally, the above proposition asserts that Hamiltonian dynamics can be as complicated as an arbitrary smooth dynamics.

A familiar subclass of Hamiltonian systems arise from the equations 
\begin{equation}\label{well-ode}
 \partial_{t}^2 u = -(\nabla_{\R^m} V)(u)
\end{equation}
of a particle in a smooth potential well $V  \colon \R^m \to \R$, where $\nabla_{\R^m} V  \colon \R^m \to \R^m$ denotes the gradient of $V$.  Indeed, by setting $q(t) \coloneqq u(t)$ and $p(t) \coloneqq \partial_t u(t)$, this ODE may be expressed as a system
\begin{equation}\label{system}
 \partial_t q = p; \quad \partial_t p = - (\nabla_{\R^m} V)(q),
\end{equation}
which is the Hamiltonian flow on the cotangent bundle
$$ T^* \R^m \coloneqq \{ (q,p): q,p \in \R^m \}$$
(with the usual symplectic form \eqref{omega}) with Hamiltonian
$$ H(q,p) \coloneqq \frac{1}{2} |p|_{\R^m}^2 + V(q)$$
where $|p|_{\R^m}$ denotes the Euclidean magnitude of $p$.  We will denote this flow as $\mathrm{Well}(\R^m, V)$.  If we assume that $V$ is \emph{coercive} in the sense that $V(q) \to +\infty$ as $q \to +\infty$, then conservation of the Hamiltonian ensures that trajectories in $\mathrm{Well}(\R^m, V)$ stay bounded, and hence global in time.

One can generalise the ODE \eqref{well-ode} to the \emph{nonlinear wave equation} (NLW)
\begin{equation}\label{nlw}
\partial_{t}^2 u - \Delta_{(\R/\Z)^d} u = -(\nabla_{\R^m} V) u
\end{equation}
where $u  \colon \R \times (\R/\Z)^d \to \R^m$ is now a smooth function of one time variable $t$ and $d$ (periodic) spatial variables $x_1,\dots,x_d$ for some $d \geq 0$ (or equivalently (by ``currying''), a smooth map from $\R$ to $C^\infty( (\R/\Z)^d \to \R^m )$), and $\Delta_{(\R/\Z)^d} = \sum_{k=1}^d \partial_{x_k}^2$ is the spatial Laplacian on $(\R/\Z)^d$.  We restrict attention here to the periodic spatial domain $(\R/\Z)^d$ to avoid technical issues relating to decay at spatial infinity.  Solutions to the potential well ODE \eqref{well-ode} can be identified with the solutions to the NLW \eqref{nlw} which are constant in the spatial variables.  Writing $q(t) \coloneqq u(t) \in C^\infty((\R/\Z)^d)$ and $p(t) \coloneqq \partial_t u(t)  \in C^\infty((\R/\Z)^d)$ as before, we can rewrite the NLW \eqref{nlw} as a first-order system
\begin{equation}\label{pqd}
 \partial_t q = p; \quad \partial_t p = \Delta_{(\R/\Z)^d} q -(\nabla_{\R^m} V)(q)
\end{equation}
which is formally a Hamiltonian flow on the infinite-dimensional phase space $C^\infty((\R/\Z)^d \to \R^m) \times C^\infty((\R/\Z)^d \to \R^m)$ with Hamiltonian
$$ H(q,p) \coloneqq \int_{(\R/\Z)^d} \frac{1}{2} |p|_{\R^m}^2 + \frac{1}{2} \sum_{k=1}^d |\partial_{x_k} q|_{\R^m}^2 + V(q)\ d\mathrm{Vol}(x)$$
where $\partial_{x_k}$ denotes the partial derivative in the $x_k$ coordinate of $(\R/\Z)^d$.  We will denote this (infinite-dimensional) system as $\mathrm{NLW}( (\R/\Z)^d, \R^m, V )$; the potential well flow $\mathrm{Well}(\R^m, V)$ then corresponds to the special case $d=0$.  While $\mathrm{NLW}( (\R/\Z)^d, \R^m, V )$ is no longer finite-dimensional for $d>0$, one can still  define the notion of an embedding of a smooth flow $(M,X)$ into $\mathrm{NLW}( (\R/\Z)^d, \R^m, V )$, namely a smooth (in the G\^ateaux sense) injective immersion $\phi$ from the manifold $M$ to the vector space $C^\infty((\R/\Z)^d \to \R^m) \times C^\infty((\R/\Z)^d \to \R^m)$, which maps trajectories of $(M,X)$ to solutions to \eqref{pqd}.  For instance, as in the finite-dimensional case, stationary solutions, periodic solutions or invariant tori for NLW are the same thing as smooth embeddings of a point, circle flow, and torus flow respectively.  Also, we may trivially embed $\mathrm{Well}(\R^m, V)$ into $\mathrm{NLW}( (\R/\Z)^d, \R^m, V )$ by mapping any point $(q,p)$ in $\R^m \times \R^m$ to the pair $(q,p) \in C^\infty((\R/\Z)^d \to \R^m) \times C^\infty((\R/\Z)^d \to \R^m)$ of constant functions from $(\R/\Z)^d$ to $q$ and $p$ respectively.

One can generalise the systems $\mathrm{Well}(\R^m,V)$ and $\mathrm{NLW}( (\R/\Z)^d, \R^m, V )$ further, by replacing the Euclidean space $\R^m$ with a more general Riemannian manifold $(M,g)$.  In this setting, $V  \colon M \to \R$ is now a smooth potential on $M$, and $\mathrm{Well}(M,V)$ is the flow on the cotangent bundle
$$ T^* M \coloneqq \{ (q,p): q \in M, p \in T^*_q M \}$$
associated to the Hamiltonian
$$ H(q,p) \coloneqq \frac{1}{2} |p|_{g(q)^{-1}}^2 + V(q)$$
where $||_{g(q)^{-1}}$ denotes the metric on $T^*_q M$ induced by the metric $g(q)$ (or more precisely, the inverse of this metric).  The equations of motion are then given by
\begin{equation}\label{goq}
 \partial_t q = g(q)^{-1} \cdot p; \quad \nabla_t p = - (d V)(q),
\end{equation}
where $g(q)^{-1} \cdot p$ is the tangent vector in $T_q M$ dual to the cotangent vector $p \in T^*_q M$ with respect to the metric $g(q)$, $dV$ is the exterior derivative of $V$, and $\nabla_t$ is the covariant derivative (using the pullback of the Levi-Civita connection by $q$); one can also write $q = u$ and $p = g(q) \cdot \partial_t u$, where $g(q) \cdot \partial_t u$ denotes the covector in $T^*_q M$ dual to the vector $\partial_t u \in T_q M$ with respect to the metric $g(q)$, and $u  \colon I \to M$ solves the second-order ODE
$$ \nabla_t \partial_t u = - (\nabla_g V)(u)$$
where $\nabla_g$ is the gradient with respect to the metric $g$.  Note that in the case $V=0$, this is just the dynamics of geodesic flow on $M$.  One can similarly define $\mathrm{NLW}((\R/\Z)^d, M, V)$ to be the formal system on $C^\infty((\R/\Z)^d \to M) \times C^\infty((\R/\Z)^d \to M)$ associated to the formal Hamiltonian
$$ H(q,p) \coloneqq \int_{(\R/\Z)^d} \frac{1}{2} |p|_{g(q)^{-1}}^2 + \frac{1}{2} \sum_{k=1}^d |\partial_{x_k} q|_{g(q)}^2 + V(q)\ d\mathrm{Vol}(x),$$
where $| |_{g(q)}$ denotes the metric on $T_q M$ induced by $g(q)$; the equations of motion are
\begin{equation}\label{qwert}
 \partial_t q = g(q)^{-1} \cdot p; \quad \partial_t p = \sum_{k=1}^d g(q) \cdot \nabla_{x_k} \partial_{x_k} q -(d V)(q)
\end{equation}
where $\nabla_{x_i}$ is the covariant derivative using the pullback of the Levi-Civita connection by $q$.  Writing $q = u$ and $p = g(q) \cdot \partial_t u$, we can also write \eqref{qwert} as a single second-order PDE
$$ \nabla_t \partial_t u - \sum_{k=1}^d \nabla_{x_k} \partial_{x_k} u = - (\nabla_g V)(u)$$
which is the equation of a wave map with potential.  One can also express this equation in coordinates using Christoffel symbols, but we will not do so here.

In this paper we study the universality properties of the class of potential well systems $\mathrm{Well}(\R^m,V)$, where we allow the number $m$ of degrees of freedom, as well as the smooth potential $V  \colon \R^m \to \R$ to be arbitrary; we also consider the analogous problem for $\mathrm{NLW}((\R/\Z)^d, \R^m, V)$, $\mathrm{Well}(M,V)$, and $\mathrm{NLW}((\R/\Z)^d, M, V)$.  It turns out that the following concept (bearing some faint resemblance\footnote{We thank Robert Bryant for this remark.} to Gromov's notion \cite{gromov} of a symplectic form that tames an almost complex structure) plays a central role:

\begin{definition}[Adapted $1$-forms]  Let $(M,X)$ be a smooth non-singular flow.  We say that a $1$-form $\theta$ on $M$ is \emph{weakly adapted} to this system if the scalar function $\theta(X)$ is everywhere non-negative and the Lie derivative ${\mathcal L}_X(\theta)$ of $\theta$ along $X$ is an exact $1$-form, thus ${\mathcal L}_X(\theta) = dL$ for some $L$.  If furthermore $\theta(X)$ is strictly positive everywhere (as opposed to merely being non-negative), we say that $\theta$ is \emph{strongly adapted} to $(M,X)$.
\end{definition}

For instance, the zero $1$-form $0$ is weakly adapted to $(M,X)$ but not strongly adapted.  The question of whether a given flow $(M,X)$ supports a strongly adapted $1$-form will end up being a key focus of this paper.

The relevance of these concepts can be seen by the following calculation.  Recall that every cotangent bundle $T^* M$ supports a \emph{canonical $1$-form} $\theta$, defined in canonical coordinates $q_1,\dots,q_m, p_1,\dots,p_m$ as $\theta \coloneqq \sum_{i=1}^m p_i dq_i$.

\begin{proposition}[Canonical form is weakly adapted]\label{not}  In the flow $\mathrm{Well}(M,V)$ (and hence also in $\mathrm{Well}(\R^m,V)$), the canonical $1$-form $\theta$ is weakly adapted to the flow. 
\end{proposition}

\begin{proof}  Let $X$ be the vector field on $T^* M$ associated to $\mathrm{Well}(M,V)$, thus from \eqref{goq} one has in coordinates that
$$ X(q,p) = (g(q)^{-1} \cdot p, -(dV)(q))$$
and hence
$$ \theta(X)(q,p) = g(q)^{-1}(p,p) = |p|_{g(q)^{-1}}^2 \geq 0.$$
On the other hand, from Cartan's formula one has
$$ {\mathcal L}_X \theta = d(\iota_X \theta) + \iota_X(d\theta)$$
where $\iota_X$ denotes contraction by $X$.  We have $\iota_X \theta = \theta(X)$ and $d\theta = -\omega$, and by Hamilton's equations of motion we have $\iota_X \omega = dH$, hence we have ${\mathcal L}_X \theta = dL$ where $L$ is the Lagrangian 
\begin{equation}\label{L-def}
 L \coloneqq \theta(X) - H = \frac{1}{2} |p|_{g(q)^{-1}}^2 - V(q),
\end{equation}
and the claim follows.
\end{proof}

\begin{remark}  The identity ${\mathcal L}_X \theta = dL$ is closely related to Noether's theorem.  Indeed, if $Y$ is a vector field that is a symmetry of the Lagrangian (in that ${\mathcal L}_Y L = 0$) and commutes with the flow, then this identity implies that ${\mathcal L}_X(\theta(Y)) = 0$, so that $\theta(Y)$ is a conserved quantity.
\end{remark}

Another key fact is that the property of being weakly or strongly adapted is preserved by pullback:

\begin{proposition}\label{pullback}  If $\phi  \colon N \to M$ is a morphism from one smooth flow $(N,Y)$ to another $(M,X)$, and $\theta$ is a $1$-form strongly adapted to $(M,X)$, then the pullback $\phi^* \theta'$ is a $1$-form strongly adapted to $(N,Y)$.  Similarly with ``strongly'' replaced by ``weakly'' throughout.
\end{proposition}

\begin{proof}  We have $(\phi^* \theta)(Y) = \phi^*( \theta(X) )$, so $(\phi^* \theta)(Y)$ is positive (resp. non-negative) if $\theta(X)$ is.  Also, for any time $t$, $(e^{tY})^* \phi^* \theta = \phi^* (e^{tX})^*\theta$; differentiating at $t=0$, we conclude that ${\mathcal L}_Y( \phi^* \theta) = \phi^*({\mathcal L}_{X} \theta)$.  Since ${\mathcal L}_{X} \theta$ is exact, ${\mathcal L}_Y \phi^* \theta$ is also.  The claim follows.
\end{proof}

These two facts suggest that the property of supporting an adapted $1$-form could serve as an obstruction to embedding into a potential well system.  Our main theorem confirms this for compact non-singular systems, and in fact shows that this is the \emph{only} obstruction in that case: 

\begin{theorem}[Criterion for embeddability]\label{main}  Let $(N,Y)$ be a compact smooth non-singular flow, and let $d \geq 0$ be an integer.  Then the following are equivalent.
\begin{itemize}
\item[(i)]  There exists $m \geq 1$ and a smooth potential $V  \colon\R^m \to \R$ such that $(N,Y)$ is embedded in $\mathrm{Well}(\R^m, V)$.
\item[(ii)]  There exists $m \geq 1$ and a smooth potential $V  \colon \R^m \to \R$ such that $(N,Y)$ is embedded in $\mathrm{NLW}((\R/\Z)^d, \R^m, V)$.
\item[(iii)]  There exists a Riemannian manifold $M$ and a smooth potential $V  \colon M \to \R$ such that $(N,Y)$ is embedded in $\mathrm{Well}(M, V)$.
\item[(iv)]  There exists a Riemannian manifold $M$ and a smooth potential $V\colon M \to \R$ such that $(N,Y)$ is embedded in $\mathrm{NLW}((\R/\Z)^d, M, V)$.
\item[(v)]  There exists a $1$-form $\theta$ strongly adapted to $(N,Y)$.
\end{itemize}
\end{theorem}

We prove this theorem in Section \ref{nash}.  The implication of (ii), (iii), or (iv) from (i) is trivial, and the implication from (v) from any of (i)-(iv) will follow from Proposition \ref{not}, Proposition \ref{pullback} and an averaging argument to upgrade the weakly adapted $1$-form to an adapted $1$-form.  To recover (i) from (v) we will use the Nash embedding theorem \cite{nash}, in a similar fashion to that in our previous paper \cite{tao-nlw}.  Informally, the equivalence of (i)-(iv) asserts that the almost periodic dynamics of nonlinear wave equations (or wave maps with potential) are no richer than the almost periodic dynamics of potential wells (either in Euclidean space or arbitrary manifolds), at least if one restricts to those dynamics generated by smooth non-singular vector fields.

\begin{remark}\label{coer}  If $(N,Y)$ is embedded in $\mathrm{Well}(V)$, then one can modify $V$ arbitrarily outside of a neighbourhood of the image of $N$ without affecting the embedding.  In particular, in the assertion (i) above one could assume without loss of generality that $V$ is coercive.  By shifting $V$ by a constant (which does not affect the dynamics) we may thus also assume without loss of generality that $V$ is non-negative.  Similarly for conclusions (ii), (iii), (iv).
\end{remark}

In view of Theorem \ref{main}, it is of interest to determine which compact smooth non-singular flows support adapted $1$-forms.  It turns out that there are many examples of flows with this property:

\begin{proposition}[Examples of strongly adapted $1$-forms]\label{examples}  Let $(N,Y)$ be a smooth non-singular flow.
\begin{itemize}
\item[(i)]  If the system $(N,Y)$ is \emph{isometric}, thus there is a Riemannian metric $g$ on $N$ which is preserved by $Y$ (that is to say, ${\mathcal L}_Y g = 0$, then the $1$-form $\theta = g \cdot Y$ that is dual to $Y$ with respect to $g$ is strongly adapted to $(N,Y)$.
\item[(ii)]  More generally, if the system $(N,Y)$ is \emph{geodesible}, thus there is a Riemannian metric $g$ on $N$ such that the trajectories of $(N,Y)$ are geodesics parameterised by arclength, then the $1$-form $\theta = g \cdot Y$ is strongly adapted to $(N,Y)$.  
\item[(iii)]  If $(N,Y)$ is an \emph{Anosov flow}, thus at every point $y \in N$, the tangent space $T_y N$ splits smoothly into the line $\R Y(y)$, the stable bundle $E^+_y$, and the unstable bundle $E^-_y$, then the canonical $1$-form $\theta$ (defined by setting $\theta(y)$ to take the value $1$ at $Y(y)$ and vanish at $E^+_y$ and $E^-_y$) is strongly adapted to $(N,Y)$.
\item[(iv)]  If $(N,Y)$ is the \emph{suspension} of some diffeomorphism $\Phi\colon M \to M$ on a compact manifold $M$, thus $N$ is the manifold formed from $M \times [0,1]$ by identifying $(y,1)$ with $(\Phi(y),0)$, with vector field $Y = (0,1)$ in the coordinate patch $M \times [0,1)$, then the $1$-form $\theta$ defined on the coordinate patch $M \times [0,1)$ by $\theta = dt$ (where $t$ denotes the second coordinate of $M \times [0,1)$) is strongly adapted to $(N,Y)$.  (Note that such suspensions will automatically be non-singular, even if the map $\Phi$ contains fixed points.)
\item[(v)]  The product system $(M \times (\R/\Z), (X,1))$ of an arbitrary smooth flow $(M,X)$ with the circle shift $(\R/\Z, 1)$ will be non-singular and has $dt$ as a strongly adapted $1$-form, where $t$ is the second coordinate on the coordinate patch $M \times [0,1)$.
\end{itemize}
\end{proposition}

\begin{proof}  We first prove (i).  Clearly $\theta(Y) = g(Y,Y)$ is positive.  Since ${\mathcal L}_Y$ annihilates both $g$ and $Y$, it annihilates $\theta = g \cdot X$, so ${\mathcal L}_Y \theta = 0$ is certainly exact.

Part (ii) is due\footnote{We are indebted to Ali Taghavi \cite{overflow} for this statement and reference.} to Sullivan \cite{sullivan} and may be proven as follows.  As in part (i), $\theta(Y) = g(Y,Y)$ is positive; in fact, because of the arclength parameterisation, we have $\theta(Y) = g(Y,Y) = 1$.  As the trajectories are geodesics, we have $\nabla_Y Y = 0$, where $\nabla$ denotes the Levi-Civita connection.  Hence, for any vector field $Z$ on $N$, we have
\begin{align*}
0 &= g(\nabla_Y Y, Z ) \\
&= \nabla_Y g(Y,Z) - g(Y, \nabla_Y Z) \\
&= {\mathcal L}_Y(g(Y,Z)) - g( Y, \nabla_Z Y ) - g( Y, [Y,Z] ) \\
&= {\mathcal L}_Y(\theta(Z)) - \frac{1}{2} \nabla_Z g(Y,Y) - \theta([Y,Z]) \\
&= {\mathcal L}_Y(\theta(Z)) - \frac{1}{2} \nabla_Z 1 - \theta({\mathcal L}_Y Z) \\
&= ({\mathcal L}_Y \theta)(Z)
\end{align*}
and hence ${\mathcal L}_Y \theta = 0$.  In \cite{sullivan} it was also noted that this calculation is reversible, thus if there exists a $1$-form $\theta$ with $\theta(Y)=1$ and ${\mathcal L}_Y \theta = 0$ then $(N,Y)$ is geodesible.

For part (iii), we have $\theta(Y)=1$ positive by construction.  The flow maps $e^{tY}$ preserves $Y$, $E^+$, and $E^-$, and thus preserves the canonical $1$-form $\theta$, hence ${\mathcal L}_Y \theta = 0$.

Part (v) is a corollary of Proposition \ref{pullback}, since we have the morphism from $N$ to the circle shift $(\R/\Z,1)$ defined by mapping $(y,t)$ to $t \hbox{ mod } 1$ for $t \in [0,1)$.  Similarly for part (vi).
\end{proof}

By Theorem \ref{main}, any of the smooth flows listed above can be embedded in a (coercive) potential well system (and hence also in a nonlinear wave equation).

In the other direction, we have the following counterexample on the $2$-torus, due to Robert Bryant:

\begin{proposition}[Bryant example]\label{thecor}  The compact non-singular smooth flow
$$
\left( (\R/\Z)^2, \sin(2\pi x) \frac{d}{dx} + \cos(2\pi x) \frac{d}{dy} \right),$$
where $x,y$ are the standard coordinates on $(\R/\Z)^2$ (see Figure \ref{fig:vec}), does not support any strongly adapted $1$-flows.  In particular (by Theorem \ref{main}), one cannot embed $(N,Y)$ into $\mathrm{Well}(\R^m, V)$, $\mathrm{Well}(M,V)$, $\mathrm{NLW}( (\R/\Z)^d, \R^m, V )$, or $\mathrm{NLW}( (\R/\Z)^d, M, V)$ for any $d \geq 0, m \geq 1$, Riemannian manifold $M$, and potential $V$.
\end{proposition}

\begin{figure} [t]
\centering
\includegraphics[height=3in]{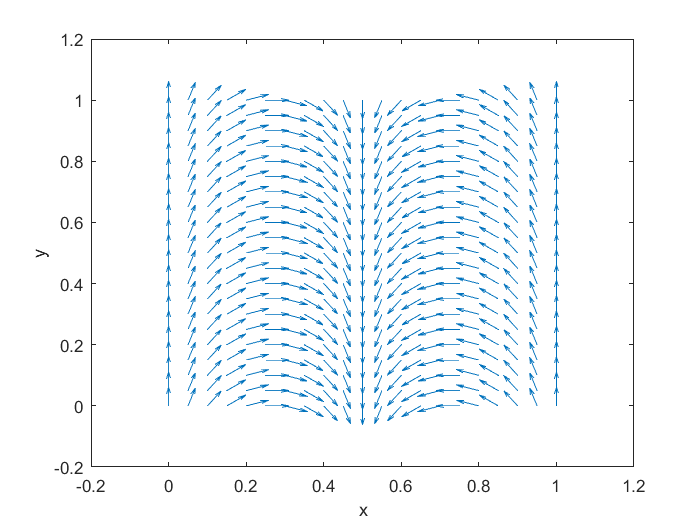}
\caption{The vector field in Proposition \ref{thecor}.  One should of course identify the $x=0,1$ edges together, as well as the $y=0,1$ edges, to obtain a vector field on the $2$-torus $(\R/\Z)^2$.}
\label{fig:vec}
\end{figure}

We reproduce Bryant's proof of this proposition in Section \ref{nonem}.  Thus we see that there are at least some almost periodic dynamics that cannot occur in a potential well or in a nonlinear wave map, and so these classes of flows are not universal.

One can also use Theorem \ref{main} (and Proposition \ref{examples}(iv)) to produce a potential well system $\mathrm{Well}(V)$ that is a universal Turing machine.  Recall that\footnote{We will use here a Turing machine with a single tape that is infinite in both directions and a single halting state, with the machine shifting the tape rather than a tape head, but the results here of course would apply to other variants of a Turing machine. It is common to isolate some special characters in the alphabet $\Sigma$, such as a blank symbol, but we will not need to do so here.} a \emph{Turing machine} consists of the following data:
\begin{itemize}
\item A finite set $Q$ of states, including an \emph{initial state} $\mathtt{START} \in Q$ and a \emph{halting state} $\mathtt{HALT} \in Q$;
\item An \emph{alphabet} $\Sigma$, which is a finite set of cardinality at least two;
\item An \emph{transition function} $\delta  \colon (Q \backslash \mathtt{HALT}) \times \Sigma \to Q \times \Sigma \times \{-1, 0, +1\}$.
\end{itemize}
 
Given a Turing machine $(Q, \mathtt{START}, \mathtt{HALT}, \Sigma, \delta)$ and an \emph{input tape} $s = (s_n)_{n \in \Z} \in \Sigma^\Z$, we can run the Turing machine by performing the following algorithm:

\begin{itemize}
\item[Step 0.]  Initialise the current state $q$ to be $\mathtt{START}$, and the current tape $t = (t_n)_{n \in \Z}$ to be $s$.
\item[Step 1.]  If $q = \mathtt{HALT}$ then halt the algorithm (and return $t$ as output).  Otherwise, compute $\delta(q, t_0) = (q', t'_0, \epsilon)$.
\item[Step 2.]  Replace $q$ with $q'$ and the $0^{\operatorname{th}}$ component $t_0$ of the tape $t$ with $t'_0$.
\item[Step 3.]  Replace the tape $t$ with the shifted tape $(t_{n-\epsilon})_{n \in \Z}$ (that is to say, perform a right shift if $\epsilon=+1$, a left shift if $\epsilon=-1$, and do nothing if $\epsilon = 0$), then return to Step 1.
\end{itemize}

Clearly, given any input $s \in \Sigma^\Z$, this Turing machine will either halt with some output $t \in \Sigma^\Z$, or run indefinitely.  

One can construct a diffeomorphism on a compact smooth manifold that is a universal Turing machine:

\begin{proposition}[Diffeomorphisms can be universal Turing machines]\label{pq}  There exists an explicitly constructible compact smooth manifold $M$ equipped with a diffeomorphism $\Phi\colon M \to M$, such that for any Turing machine $(Q, \mathtt{START}, \mathtt{HALT}, \Sigma, \delta)$ there exists an explicitly constructible open set $U_{t_{-n},\dots,t_n} \subset M$ attached to each finite string $t_{-n},\dots,t_n \in \Sigma$, and an explicitly constructible point $y_s \in M$ attached to each $s \in \Sigma^\Z$, such that the Turing machine $(Q, \mathtt{START}, \mathtt{HALT}, \Sigma, \delta)$ with input tape $s$ halts with output tape having coefficients $t_{-n},\dots,t_n$ in positions $-n,\dots,n$ respectively if and only if the orbit $y_s, \Phi(y_s), \Phi^2(y_s), \dots$ enters $U_{t_{-n},\dots,t_n}$ (that is, $\Phi^m(y_s) \in U_{t_{-n},\dots,t_n}$ for some $m$).
\end{proposition}

This claim is standard (and not surprising, given the close relationship between smooth dynamics and symbolic dynamics); we establish it in Section \ref{turing}.  Combining this with Theorem \ref{main} and Proposition \ref{examples}(iv), we conclude

\begin{corollary}[Potential wells can be universal Turing machines]\label{turam}  There exists a coercive potential $V \colon \R^m \to \R$ such that for any Turing machine $(Q, \mathtt{START}, \mathtt{HALT}, \Sigma, \delta)$ there exists an explicitly constructible open set $U_{t_{-n},\dots,t_n} \subset \R^m \times \R^m$ attached to each finite string $t_{-n},\dots,t_n \in \Sigma$, and an explicitly constructible (and bounded) point $y_s \in \R^m \times \R^m$ attached to each $s \in \Sigma^\Z$, such that the Turing machine $(Q, F, q_0, \Sigma, \delta)$ with input $s$ halts with output 
tape having coefficients $t_{-n},\dots,t_n$ in positions $-n,\dots,n$ respectively if and only if the trajectory in $\mathrm{Well}(\R^m,V)$ with initial data $(q_s,p_s)$ enters $\tilde U_{t_{-n},\dots,t_n}$ at some non-negative time.
\end{corollary}

\begin{proof}  Let $\phi\colon M \to M$ be the diffeomorphism from Proposition \ref{pq}, and let $(\tilde M, X)$ be the suspension of $\phi$.  By Proposition \ref{examples}(iv), $(\tilde M, X)$ is compact, non-singular, and supports an strongly adapted $1$-form, and hence by Theorem \ref{main} it may be embedded in $\mathrm{Well}(\R^m, V)$ for some $m, V$, and by Remark \ref{coer} we may make $V$ coercive.  An inspection of Theorem \ref{main} shows that the embedding can be explicitly constructed (using for instance the Nash embedding construction\footnote{This construction involves solving some explicit elliptic (and slightly non-local) PDE, and so for the purposes of this paper one needs to view the solution of such PDE (which can be done for instance by performing a Picard iteration and then taking limits) as an ``explicit construction''.  One may also argue that earlier proofs of the embedding theorem that rely on Nash-Moser iteration also yield explicit constructions.} from \cite{gunther}).  The claim follows by taking $(q_s,p_s)$ to be the image of $(y_s,0)$ under this embedding, and $\tilde U_{t_{-n},\dots,t_n}$ to be (a neighbourhood of) the image of $U_{t_{-n},\dots,t_n} \times \{0\}$.
\end{proof}

Given that the halting problem is undecidable, we conclude in particular that there exist explicitly constructable potential well systems $\mathrm{Well}(\R^m, V)$ and an explicitly constructible trajectory in that system, such that it is undecidable whether that trajectory enters an explicit open set $U$ at some non-negative time.  As another special case, we may construct explicit trajectories which enter such an open set if and only if there is a counterexample to (say) the Riemann hypothesis, by constructing a suitable Turing machine to look for such counterexamples (using for instance Lagarias's formulation \cite{lagarias} $\sum_{d|n} d \leq H_n + \exp(H_n) \log(H_n)$ of that hypothesis, where $H_n = \sum_{i=1}^n \frac{1}{i}$ are the harmonic numbers); similarly for many other unsolved problems in mathematics.  Informally, we conclude that the dynamics of an arbitrary potential well system can be arbitrarily complicated.  Of course, the same also holds for the nonlinear wave equation.

\begin{remark}  In \cite{tao-navier}, the author speculated that if one could demonstrate that the Euler equations were Turing-complete, this could be used to create a solution to the Navier-Stokes equations that exhibited finite time blowup by creating initial data that is ``programmed'' to evolve to a rescaled version of itself (up to some hopefully negligible errors).  One can view Corollary \ref{turam} as establishing an analogous Turing-completeness for a nonlinear wave equation (although blowup for such equations was already demonstrated in \cite{tao-nlw}, at least in the case of three spatial dimensions).
\end{remark}

\begin{remark}  There are other results in the literature establishing that certain flows or maps can be universal Turing machines.  For instance, in \cite{gcb}, an analytic map on a non-compact manifold was constructed which could serve as a (robust) universal Turing machine, while in \cite{kcg} a piecewise linear continuous map was constructed which also served as a universal Turing machine.
\end{remark}

The author is supported by NSF grant DMS-1266164 and by a Simons Investigator Award.  We thank Sungjin Oh and Khang Hunyh for helpful conversations and corrections, to Ali Taghavi \cite{overflow} for pointing the author towards the reference \cite{sullivan}, and to Robert Bryant \cite{overflow} for supplying the counterexample in Proposition \ref{thecor}.  

\section{Proof of main theorem}\label{nash}

We now prove Theorem \ref{main}.  As noted previously, it is immediate from the constant embedding of $\mathrm{Well}(\R^m, V)$ in $\mathrm{NLW}((\R/\Z)^d, \R^m, V)$ that (i) implies (ii); similarly, (iii) implies (iv).  It is also trivial that (i) implies (iii), and that (ii) implies (iv).

Now we show that (iii) implies (v); this will be made redundant later when we show that (iv) also implies (v), but this simpler implication serves to motivate the argument in the latter case.  

We need the following simple averaging trick to upgrade weakly adapted forms to strongly adapted ones:

\begin{lemma}[Averaging argument]\label{avg}  Let $(N,Y)$ be a compact non-singular smooth flow.  Suppose that $\theta$ is a $1$-form weakly adapted to $(N,Y)$, with the property that $\theta(Y)$ does not vanish on any arc of the form $\{ e^{tY} y: 0 \leq t \leq T \}$ with $y \in N$ and $T>0$.  Then there exists another $1$-form $\tilde \theta$ which is strongly adapted to $(N,Y)$.
\end{lemma}

\begin{proof}  For any time $t$, the flow $e^{tY}$ preserves the vector field $Y$ and commutes with ${\mathcal L}_Y$.  As $\theta$ is weakly adapted to $(N,Y)$, we conclude that the pullbacks $(e^{tY})^* \theta$ are also weakly adapted to $(N,Y)$, and by linearity we conclude that the average $\int_0^1 (e^{tY})^* \theta\ dt$ is also weakly adapted.  However, since $\theta(Y)$ does not vanish on any arc, the quantity
$$ \left(\int_0^1 (e^{tY})^* \theta\ dt\right)(Y) = \int_0^1 (e^{tY})^* (\theta(Y))\ dt$$
never vanishes, and the claim follows.
\end{proof}

From the hypothesis (iii), we have a smooth potential $V\colon M \to \R$ on a Riemannian manifold $M$ and an embedding $\phi \colon N \to T^* M$ of $(N,Y)$ into $\mathrm{Well}(M,V)$. By Proposition \ref{not}, the canonical $1$-form $\theta$ on $T^* M$ is weakly adapted to $\mathrm{Well}(M,V)$.
By Proposition \ref{pullback}, the pullback $\phi^* \theta$ is then weakly adapted to $(N,Y)$.  By Lemma \ref{avg}, we can conclude (v) unless $(\phi^* \theta)(Y)$ vanishes on some arc $\{ e^{tY} y: 0 \leq t \leq T \}$ with $y \in N$ and $T > 0$.  Suppose for contradiction that we have such a vanishing.  If we write $(p(t),q(t)) = \phi( e^{tY} y)$ for $t \in \R$, then (by the definition of the canonical $1$-form $\theta$) $(p,q)$ is a trajectory in $\mathrm{Well}(M,V)$ with the property that $p(t)( \partial_t q(t) )$ vanishes for $0 \leq t \leq T$.  From \eqref{goq} we have
$$ p(t)( \partial_t q(t) ) = |p(t)|_{g(q(t))^{-1}}^2$$
and hence $p(t)$ vanishes for $0 \leq t \leq T$, and hence $\partial_t q(t) = p(t)$ vanishes also.  In particular, $\partial_t (p(t), q(t)) = d\phi( Y(y) )$ vanishes at $t=0$, which contradicts the fact that $\phi$ is an immersion and that $Y$ is non-vanishing at $y$.  This proves that (iii) implies (v).

For future reference, we observe that the pullback $\phi^* \theta$ used in the above argument can be expressed using local canonical coordinates $q_1,\dots,q_m, p_1,\dots,p_m$ for $M$ and local coordinates $y_1,\dots,y_n$ for $N$ (where $m,n$ are the dimensions of $M,N$ respectively) as
\begin{equation}\label{qf}
\begin{split}
 \phi^* \theta(y) &= \sum_{i=1}^m (p_i \circ \phi)(y) d(q_i \circ \phi)(y)\\
&= \sum_{i=1}^m \sum_{j=1}^n p_i(\phi(y)) \partial_{y_j}(q_i \circ \phi)(y) dy_j \\
&= \sum_{j=1}^n p(\phi(y)) ( \partial_{y_j}(q \circ \phi)(y) ) dy_j.
\end{split}
\end{equation}

Now we show that (iv) implies (v).  This argument is similar to the previous one, but in order to avoid performing any differential geometry on an infinite dimensional manifold, we will use more explicit computations in coordinates than before.

By hypothesis, we have a smooth potential $V\colon M \to \R$ on a Riemannian manifold $M$ and an embedding $\phi \colon N \to T^* M$ of $(N,Y)$ into $\mathrm{NLW}((\R/\Z)^d,M,V)$.  We write $\phi = (Q, P)$, where for each $y \in N$ and $x \in (\R/\Z)^d$, $Q(y,x) = (q \circ \phi)(y)(x)$ is a point in $M$, and $P(y,x) = (p \circ \phi)(y)(x)$ is a cotangent vector in $T_{Q(y,x)} M$, with $P$ and $Q$ varying smoothly in both the $x$ and $y$ variables.

The analogue of the pullback form $\phi^* \theta$ used in the previous argument will be given in local coordinates $y_1,\dots,y_n$ for $N$ by the formula
\begin{equation}\label{tty}
\tilde \theta \coloneqq \sum_{j=1}^n \left(\int_{(\R/\Z)^d} P(y,x) ( \partial_{y_j} Q(y,x) )\ d\mathrm{Vol}(x)\right)\ dy_j.
\end{equation}
where $d\mathrm{Vol}$ is the standard volume form on $(\R/\Z)^d$. It is easy to see that this does not depend on the choice of local coordinates $y_1,\dots,y_n$, so that $\tilde \theta$ is indeed a $1$-form.

It is convenient to work in local coordinates $y_1,\dots,y_n$ for which the vector field $Y$ is just $\frac{d}{dy_n}$, so that the Lie derivative ${\mathcal L}_Y$ is just $\partial_{y_n}$; such a coordinate system is always locally available as $Y$ is non-singular.  In these coordinates, we see from \eqref{qwert} that we have the equations of motion
\begin{align*}
\partial_{y_n} Q &= g(Q)^{-1} \cdot P \\
\nabla_{y_n} P &= \sum_{i=1}^d g(Q) \cdot \nabla_{x_i} \partial_{x_i} Q -(d V)(Q),
\end{align*}
where $\nabla$ is the pullback of the Levi-Civita connection by $Q$, and we suppress the variables $y,x$ for brevity.  In particular, we have
\begin{align*}
\tilde \theta(Y) &= \int_{(\R/\Z)^d} P ( \partial_{y_n} Q )\ d\mathrm{Vol}(x) \\
&= \int_{(\R/\Z)^d} | P|_{g(Q)^{-1}}^2\ d\mathrm{Vol}(x)
\end{align*}
and hence $\tilde \theta(Y)$ is always non-negative.  Furthermore, the only way that $\tilde \theta(Y)$ could vanish on an small arc $\{ e^{tY} y: 0 \leq t \leq T \} =  \{ y + t e_n: 0 \leq t \leq T\}$ in these local coordinates is if $P(y+te_n,x)$ vanished for all $0 \leq t \leq T$ and $x \in (\R/\Z)^d$, which by the equations of motion show that $\partial_{y_n} Q(y+te_n,x)$ vanished also; thus the map $t \mapsto \phi(y + t e_n)$ from $[0,T]$ to $C^\infty((\R/\Z)^d \to T^* M)$ is stationary at $t=0$, contradicting the hypothesis that $\phi$ is an embedding.  Thus $\tilde \theta(Y)$ does not vanish on any such arc.

Now we compute the Lie derivative ${\mathcal L}_Y \tilde \theta$.   In local coordinates this is $\partial_{y_n} \tilde \theta$.
From the Leibniz rule, the Lie derivative in these coordinates becomes 
$$
{\mathcal L}_Y \tilde \theta = \sum_{j=1}^n \left(\int_{(\R/\Z)^d} (\nabla_{y_n} P)( \partial_{y_j} Q) + P( \nabla_{y_n} \partial_{y_j} Q )\ d\mathrm{Vol}(x)\right)\ dy_j.$$
As the Levi-Civita connection is torsion-free, $\nabla_{y_n} \partial_{y_j} Q$ is equal to $\nabla_{y_j} \partial_{y_n} Q$.
Using the equations of motion, the above expression then becomes
\begin{align*}
&\sum_{j=1}^n \sum_{i=1}^d \left(\int_{(\R/\Z)^d} (g(Q) \cdot \nabla_{x_i} \partial_{x_i} Q)( \partial_{y_j} Q)\ d\mathrm{Vol}(x)\right)\ dy_j \\
& \quad - \sum_{j=1}^n \left(\int_{(\R/\Z)^d} (dV(Q))(\partial_{y_j} Q)\ d\mathrm{Vol}(x)\right)\ dy_j \\
& \quad + \sum_{j=1}^n \left(\int_{(\R/\Z)^d} P( \nabla_{y_j} (g(Q)^{-1} \cdot P) )\ d\mathrm{Vol}(x)\right)\ dy_j.
\end{align*}
The first term can be rewritten as
$$ \sum_{j=1}^n \left(\sum_{i=1}^d \int_{(\R/\Z)^d} g(Q)( \nabla_{x_i} \partial_{x_i} Q, \partial_{y_j} Q )\ d\mathrm{Vol}(x)\right)\ dy_j $$
which after integration by parts (recalling that the Levi-Civita connection is parallel to the metric $g$) becomes
$$ - \sum_{j=1}^n \left(\sum_{i=1}^d \int_{(\R/\Z)^d} g(Q)( \partial_{x_i} Q, \nabla_{x_i} \partial_{y_j} Q )\ d\mathrm{Vol}(x)\right)\ dy_j.$$
Using the torsion-free nature of the Levi-Civita connection, this is
$$ - \sum_{j=1}^n \left(\sum_{i=1}^d \int_{(\R/\Z)^d} g(Q)( \partial_{x_i} Q, \nabla_{y_j} \partial_{x_i} Q )\ d\mathrm{Vol}(x)\right)\ dy_j$$
which since the Levi-Civita connection is parallel to $g$, becomes
$$ - \frac{1}{2} \sum_{j=1}^n \partial_{y_j} (\sum_{i=1}^d \int_{(\R/\Z)^d} g(Q)( \partial_{x_i} Q, \partial_{x_i} Q )\ d\mathrm{Vol}(x))\ dy_j.$$
This is an exterior derivative and is thus exact.  Similarly, the second term
$$ - \sum_{j=1}^n \left(\int_{(\R/\Z)^d} (dV(Q))(\partial_{y_j} Q)\ d\mathrm{Vol}(x)\right)\ dy_j $$
can be written using the chain rule as an exterior derivative
$$ - \sum_{j=1}^n \partial_{y_j} \left(\int_{(\R/\Z)^d} V(Q)\ d\mathrm{Vol}(x)\right)\ dy_j$$
and is thus also exact.  Finally, the third term
$$\sum_{j=1}^n \left(\int_{(\R/\Z)^d} P( \nabla_{y_j} (g(Q)^{-1} \cdot P) )\ d\mathrm{Vol}(x)\right)\ dy_j$$
can be written using the Leibniz rule and the fact that the Levi-Civita connection is parallel to $g$ as yet another exterior derivative
$$ \frac{1}{2} \sum_{j=1}^n \partial_{y_j} \left(\int_{(\R/\Z)^d} |P|^2_{g(Q)^{-1}}\ d\mathrm{Vol}(x)\right)\ dy_j$$
and is also exact.  Thus ${\mathcal L}_Y \tilde \theta$ is exact.  Indeed, we have shown the identity
$${\mathcal L}_Y \tilde \theta = dL$$
where $L \colon N \to \R$ is the spatially integrated Lagrangian
$$ L \coloneqq \int_{(\R/\Z)^d} \frac{1}{2} |P|^2_{g(Q)^{-1}} - \sum_{i=1}^d |\partial_{x_i} Q|_{g(Q)}^2 - V(Q)\ d\mathrm{Vol}(x);$$
this should be compared with the proof of Proposition \ref{not}.  

From the above discussion we see that $\tilde \theta$ is weakly adapted to $(N,Y)$ with $\tilde \theta(Y)$ not vanishing identically on any arc, and so the claim (v) follows from Lemma \ref{avg} as before.

\begin{remark}  In the above calculation, one could have replaced the torus $(\R/\Z)^d$ with any other compact Riemannian manifold (replacing the volume form $d\mathrm{Vol}$ by the Riemannian measure), albeit at the cost of having some rather confusing notation to treat the three different Riemannian manifolds that are now involved; we leave the details to the interested reader.
\end{remark}

Finally, we show that (v) implies (i). By hypothesis, we have a smooth $1$-form $\theta$ on $N$ and a smooth function $L \colon N \to \R$ such that \begin{equation}\label{thetaf}
{\mathcal L}_Y \theta = dL
\end{equation}
and such that $\theta(Y)$ is strictly positive.  By compactness, $\theta(Y)$ is bounded away from zero.

The first step is to find an embedding $(q,p) \colon N \to \R^m \times \R^m$, with $p = {\mathcal L}_Y q$, such that the pullback of the canonical $1$-form $\sum_{i=1}^m p_i dq_i$ by $(q,p)$ is equal to $\theta$.  Our main tool for doing this will be the Nash embedding theorem.

We place an arbitrary smooth Riemannian metric $g$ on $N$.  We define a new metric $\tilde g$ by the formula
$$ \tilde g(aY + Z, bY + W) \coloneqq ab \theta(Y) + a \theta(W) + b \theta(Z) + C g(Z,W) $$
whenever $a,b \in \R$ and $Z,W$ are orthogonal to $Y$ (with respect to $g$), where $C>0$ is a large constant to be chosen later.  This is clearly a symmetric $2$-tensor, and
$$ \tilde g(aY+Z, aY+Z) = a^2 \theta(Y) + 2a \theta(Z) + C g(Z,Z)$$
whenever $a \in \R$ and $Z$ is orthogonal to $Y$ (with respect to $g$).  Since $\theta(Y)$ is bounded away from zero, we see that $\tilde g$ is positive definite if $C$ is large enough, so that $(M, \tilde g)$ is a Riemannian manifold.  Also, we see from construction that $\tilde g( Z, X ) = \theta(Z)$ for all vector fields $Z$, thus $\theta$ and $Y$ are duals of each other with respect to $\tilde g$.

We now apply the Nash embedding theorem \cite{nash}.  This produces a smooth isometric embedding $q \colon N \to \R^m$ from $(N,\tilde g)$ to a Euclidean space, thus $q$ is a smooth injective immersion such that
$$ \langle \partial_{y_i} q, \partial_{y_j} q \rangle_{\R^m} = \tilde g( e_i, e_j )$$
in local coordinates $y_1,\dots,y_n$ for all $i=1,\dots,n$, where $\langle, \rangle_{\R^m}$ is the Euclidean inner product on $\R^m$.  In particular (using coordinates in which $Y = \frac{d}{dy_n}$) we have
$$ \langle \partial_{y_i} q, \partial_{y_n} q \rangle_{\R^m} = \theta\left( \frac{d}{dy_i} \right)$$
for $i=1,\dots,n$; if we then define $p\colon M \to \R^m$ in coordinates to be 
$$p \coloneqq \partial_{y_n} q$$
then we see that 
\begin{equation}\label{theta-form}
 \theta = \sum_{i=1}^n \langle p, \partial_{y_i} q \rangle_{\R^m} \frac{d}{dx_i}
\end{equation}
in coordinates.  In coordinate-free notation, we have $p = {\mathcal L}_Y q$, and $\theta$ is the pullback of the canonical $1$-form by $(q,p)$.

Since the map $q\colon N \to\R^m$ was already a smooth injective immersion, and $p\colon N \to\R^m$ is smooth, the map $\phi\colon N \to\R^m \times \R^m$ defined by $\phi(y) \coloneqq (q(y), p(y))$ is also a smooth injective immersion.  To conclude (i), it suffices to locate a smooth potential $V \colon \R^m \to \R$ so that $\phi$ is a morphism from $(N,Y)$ to $\mathrm{Well}(\R^m, V)$.  By \eqref{system}, this amounts to verifying the equations of motion
\begin{align}
{\mathcal L}_Y q &= p \label{yp-0}\\
{\mathcal L}_Y p &= - (\nabla_{\R^m} V)(q) \label{yp-1}.
\end{align}
The first equation is already verified, so we work on the second.  Again, we work in local coordinates for which $Y = \frac{d}{dy_n}$, so that ${\mathcal L}_Y$ is just the partial derivative $\partial_{y_n}$.

From \eqref{thetaf}, \eqref{theta-form} we have
$$ \partial_{y_n} \langle p, \partial_{y_i} q \rangle_{\R^m} = \partial_{y_i} L$$
for all $i=1,\dots,n$.  We now let $v\colon N \to\R$ be the smooth function such that
$$ L = \frac{1}{2} |p|_{\R^m}^2 - v$$
where $| |_{\R^m}$ denotes the Euclidean norm on $\R^m$; comparing with \eqref{L-def}, we see that $v$ is ``supposed'' to be $V \circ q$.  We now compute
\begin{align*}
\partial_{y_n} \langle p, \partial_{y_i} q \rangle_{\R^m} &= \partial_{y_i} L \\
&= \langle p, \partial_{y_i} p \rangle_{\R^m} - \partial_{y_i} v \\
&= \langle p, \partial_{y_n} \partial_{y_i} q \rangle_{\R^m}
\end{align*}
and hence by the Leibniz rule
\begin{equation}\label{uouo}
 \langle \partial_{y_n} p, \partial_{y_i} q \rangle_{\R^m} = - \partial_{y_i} v.
\end{equation}
Let $q(N) \subset \R^m$ be the image of $N$ under the smooth injective immersion $q$: this is a compact $n$-dimensional submanifold of $\R^m$, with a smooth inverse map $q^{-1} \colon q(N) \to N$.  On $q(N)$, we define the acceleration field $a \colon q(N) \to \R^m$ and the restricted potential field $V_0 \colon q(N) \to \R^m$ by the formulae
$$ a \coloneqq \partial_{y_n} p \circ q^{-1}$$
and
$$ V_0 \coloneqq v \circ q^{-1}.$$
At any point $q(y)$ of $q(N)$, we see from \eqref{uouo} and the chain rule that
$$ \langle a(q(y)), \partial_{y_i} q(y) \rangle_{\R^m} = - \langle \nabla_{q(N)} V_0(q(y)), \partial_{y_i} q(y) \rangle_{T_{q(y)} q(N)}$$
where $T_{q(y)} q(N)$ is the tangent space to $q(N)$ at $q(y)$, viewed as a subspace of $\R^m$ with the induced inner product (and noting that $\partial_{y_i} q(y)$ lies in $T_{q(y)} q(N)$), and $\nabla_{q(N)}$ is the gradient operator associated to the submanifold $q(N)$ of the Euclidean space $\R^m$.

As $q$ is am immersion, the tangent vectors $\partial_{y_1} q(y), \dots, \partial_{y_n} q(y)$ form a basis for $T_{q(y)} q(N)$.  We conclude that
$$ a(z) = - \nabla_{q(N)} V_0(z) + n(z) $$
for all $z \in q(N)$, where $n(z)$ is a vector in $\R^m$ orthogonal to the tangent space $T_{q(y)} q(N)$ and varying smoothly in $z$.  Using Fermi normal coordinates around the smooth compact submanifold $q(N)$ of $\R^m$, we may thus find a smooth function $V \colon {\mathcal N}_\eps(q(N)) \to \R$ on a tubular neighbourhood ${\mathcal N}_\eps(q(N))$ of $q(N)$ which extends the function $V_0\colon N \to\R$, and is such that
$$ a(z) = - \nabla_{\R^m} V(z)$$
for all $z \in q(N)$.  By multiplying $V$ by a smooth cutoff function supported on ${\mathcal N}_\eps(q(N))$ and equal to $1$ on a smaller neighbourhood of $q(N)$, we may assume without loss of generality that $V$ extends smoothly to a (compactly supported) potential $V \colon\R^m \to \R$.  From the definition of $a$, we now have
$$ \partial_{y_n} p = -(\nabla V)(q)$$
on all of $N$, giving the required equation of motion \eqref{yp-1}. This concludes the implication of (i) from (v), and the proof of Theorem \ref{main} is complete.

\begin{remark}  Using the version of the Nash embedding theorem by Gunther \cite{gunther}, one can take the dimension $m$ of the potential well to be $\max(n(n=5)/2, n(n+3)/2 +5)$.
\end{remark}

\section{A flow without a strongly adapted $1$-form}\label{nonem}

We now present the argument of Bryant \cite{overflow} that proves Proposition \ref{thecor}.  Let $Y$ denote the vector field
$$ Y \coloneqq \sin(2\pi x) \frac{d}{dx} + \cos(2\pi x) \frac{d}{dy} $$
on the $2$-torus $(\R/\Z)^2$.  This is clearly a compact non-singular smooth flow.  Suppose for contradiction that we could find a $1$-form $\theta$ on this torus with $\theta(Y)$ positive and 
$${\mathcal L}_Y \theta = dL$$
for some smooth $L \colon (\R/\Z)^2 \to \R$.  By Cartan's formula, we have
$${\mathcal L}_Y \theta = d(\theta(Y)) + \iota_Y(d\theta)$$
where $\iota_Y$ denotes contraction by $Y$.  If we then define the ``Hamiltonian'' $H \coloneqq \theta(Y) - L$, we thus have
\begin{equation}\label{dh}
 dH = - \iota_Y(d\theta).
\end{equation}
Contracting this against $Y$ once more, we conclude that ${\mathcal L}_Y H = 0$, thus $H$ is constant along trajectories of the flow.

A trajectory $t \mapsto (x(t), y(t))$ of the flow solves the system of ODE
\begin{align*}
\partial_t x(t) &= \sin(2\pi x(t)) \\
\partial_t y(t) &= \cos(2\pi x(t)).
\end{align*}
The first ODE $\partial_t x(t) = \sin(2\pi x(t))$ has two fixed point solutions in $\R/\Z$: the repelling fixed point $x(t) = 0 \hbox{ mod } 1$ and the attracting fixed point $x(t) = 1/2 \hbox{ mod } 1$.  An inspection of the sign pattern of $\sin(2\pi x)$ reveals that all other solutions to this ODE go to $0 \hbox{ mod } 1$ as $t \to -\infty$ and to $1/2 \hbox{ mod } 1$ as $t \to +\infty$.  If we define the invariant circles
\begin{align*}
C_0 &\coloneqq \{ 0 \hbox{ mod } 1 \} \times \R/\Z \\
C_1 &\coloneqq \{ 1/2 \hbox{ mod } 1 \} \times \R/\Z 
\end{align*}
we conclude that the trajectories to the flow $((\R/\Z)^2, Y)$ either stay within $C_0$, stay within $C_1$, or else approach $C_0$ (oscillating infinitely often in the $y$ direction) as $t \to -\infty$ and approach $C_1$ (again oscillating infinitely often) as $t \to +\infty$ (cf. Figure \ref{fig:vec}).  In particular, as $H$ is continuous and constant along trajectories, $H$ must be constant on $C_0$, and the value of $H$ on any other trajectory must equal its value at $C_0$, and hence $H$ is constant on the entire $2$-torus.  From \eqref{dh} we conclude that $\iota_Y(d\theta) = 0$; since $d\theta$ is a $2$-form on a two-dimensional manifold, and $Y$ never vanishes, we conclude that $d\theta$ must vanish identically.  By Stokes theorem, this implies that
$$ \int_{C_0} \theta = \int_{C_1} \theta $$
where we orient both $1$-cycles $C_0, C_1$ in the forward $y$ direction.  But $Y$ is equal to $(0,1)$ on $C_0$ and $(0,-1)$ on $C_1$, hence
$$ \int_{\R/\Z} \theta(Y)(0 \hbox{ mod } 1, y)\ dy= - \int_{\R/\Z} \theta(Y)(1/2 \hbox{ mod } 1, y)\ dy$$
which is inconsistent with $\theta(Y)$ being everywhere positive.  The claim follows.

\section{Encoding a Turing machine}\label{turing}

We now prove Proposition \ref{pq}.  It will suffice show Proposition \ref{pq} for a single Turing machine, namely a \emph{universal} Turing machine (see e.g. \cite[\S 1.4]{arora}), since by definition this machine can be used to model all other Turing machines.

Thus, let us now fix a universal Turing machine $(Q,\mathtt{START}, \mathtt{HALT}, \Sigma, \delta)$.  The running state of such a machine is described by a state $q \in Q$ and a tape $t \in \Sigma^\Z$.  The state space $Q$ is already a (zero-dimensional) compact smooth manifold, but the tape space $\Sigma^\Z$ is not.  However, this is easily fixed via a suitable embedding.  Firstly, without loss of generality we may suppose that $\Sigma = \{0,1,\dots,k\}$ for some natural number $k \geq  1$.  Let $b$ be a base much larger than $k$ (e.g. $b=10k$ will suffice).  
We then create an embedding $f \colon \Sigma^\Z \to (\R/\Z)^2$ into the $2$-torus $(\R/\Z)^2$ by defining
$$ f( (t_n)_{n \in \Z} ) \coloneqq \left( \sum_{n=1}^\infty t_n b^{-n} \hbox{ mod } 1, \sum_{n=1}^\infty t_{1-n} b^{-n} \hbox{ mod } 1 \right)$$
whenever $t_n \in \{0,\dots,k\}$ for $n \in \Z$.  The image $f(\Sigma^\Z)$ is thus the product of two Cantor sets.  We claim that there is a diffeomorphism $\phi \colon (\R/\Z)^2 \to (\R/\Z)^2$ that encodes the right shift in the sense that
\begin{equation}\label{fat}
 f( (t_{n-1})_{n \in \Z} ) = \phi( f( (t_{n})_{n \in \Z} ) )
\end{equation}
whenever $t_n \in \{0,\dots,k\}$ for $n \in \Z$.  Indeed, for any $j = 0,\dots,k$, define the rectangles $R_j, S_j \subset (\R/\Z)^2$ by the formulae
\begin{align*}
R_j &\coloneqq \left[0, \frac{k+1}{b}\right] \times \left[\frac{j}{b}, \frac{j}{b} + \frac{k+1}{b^2}\right] \hbox{ mod } \Z^2 \\
S_j &\coloneqq \left[\frac{j}{b}, \frac{j}{b} + \frac{k+1}{b^2}\right] \times \left[0, \frac{k+1}{b}\right] \hbox{ mod } \Z^2.
\end{align*}
For $b$ large enough, $R_0,\dots,R_k$ are disjoint rectangles in $(\R/\Z)^2$, and similarly for $S_0,\dots,S_k$.  One can then construct a diffeomorphism $\phi$ that maps each $R_j$ (affine-)linearly to $S_j$ by the formula
$$ \phi\left( \alpha, \frac{j}{b} + \frac{\beta}{b} \right) \coloneqq \left( \frac{j}{b} + \frac{\alpha}{b}, \beta \right)$$
for all $\alpha,\beta \in [0, \frac{1}{b-1}]$, and maps $(\R/\Z)^2 \backslash (R_0 \cup \dots \cup R_k)$ smoothly to $(\R/\Z)^2 \backslash (S_0 \cup \dots \cup S_k)$ in some arbitrary fashion; this is possible because one can smoothly deform the closure of $(\R/\Z)^2 \backslash (R_0 \cup \dots \cup R_k)$ to the closure of $(\R/\Z)^2 \backslash (S_0 \cup \dots \cup S_k)$ while mapping the boundary of each $R_j$ to the corresponding boundary of $S_j$ without any rotation.  One can then check \eqref{fat} by direct computation.

To each state $q \in Q$ we associate a closed square $B_q$ in $(\R/\Z)^2$, such that the $B_q$ are all disjoint (we need two-dimensions here to prevent the complement of the union of $\bigcup_{q \in Q} B_q$ from being disconnected).  Our manifold $M$ will
then be the $4$-torus
$$ M \coloneqq (\R/\Z)^2 \times (\R/\Z)^2.$$
To each state $s \in \Sigma^\Z$, the starting point $y_s \in M$ is then defined by the formula
$$ y_s \coloneqq ( x_{\mathtt{START}}, f(s) )$$
where $x_{\mathtt{START}}$ is the centre of $B_{\mathtt{START}}$,
and the open set $U_{t_{-n},\dots,t_n} \subset M$ will be defined as
$$U_{t_{-n},\dots,t_n} \coloneqq V \times W_{-t_n,\dots,t_n},$$
where $V \subset (\R/\Z)^2$ is any open neighbourhood of the square $B_{\mathtt{HALT}}$ that does not intersect any other square $B_q$, and $W_{-t_n,\dots,t_n} \subset (\R/\Z)^2$ is any open neighbourhood of the two-dimensional Cantor set
$$ \{ f( (t'_m)_{m \in \Z} ): t'_m = t_m \hbox{ for all } m=-n,\dots,n \} $$
that does not contain any other point of the Cantor set $f(\Sigma^\Z)$.  Finally, the diffeomorphism $\Phi\colon M \to M$ is defined as follows.  
For each $q \in Q \backslash \{ \mathtt{HALT} \}$ and $t_0 \in \{0,\dots,k\}$, write
\begin{equation}\label{qat}
 \delta(q,t_0) = (q', t'_0, \epsilon),
\end{equation}
and let $B'_{q,t_0}$ be a closed small ball contained in $B_{q'}$, such that the $B'_{q,t_0}$ are disjoint as $q,t_0$ vary.  Let $L_{q,t_0} \colon B_q \to B'_{q,t_0}$ be the homothety that maps $B_q$ diffeomorphically onto $B'_{q,t_0}$.  On the four-dimensional box 
$$ B_q \times R_{t_0},$$
we define $\Phi$ to be the map
$$ \Phi( z, w ) \coloneqq \left( L_{q,t_0}(z), \phi^\epsilon\left(w + (0, - \frac{t_0}{b} + \frac{t'_0}{b})\right) \right)$$
for $z \in B_q$ and $w \in R_{t_0}$.  Because the squares $B'_{q,t_0}$ are disjoint, the images $\Phi(B_q \times R_{t_0})$ are disjoint (and diffeomorphic to four-dimensional boxes) as $q,t_0$ vary.  By smoothly deforming the complement of these images back to the original local $B_q \times R_{t_0}$ (which is possible due to the connected nature of the complement and the contractible nature of the boxes $B_q \times R_{t_0}$), we can then extend $\Phi$ to be a diffeomorphism on all of $M$.

By construction, if a point $(z,w)$ is such that $z \in B_q$ and $w = f( (t_n)_{n \in \Z} )$ for some $q \in Q \backslash \{ \mathtt{HALT} \}$ and $t_n \in \{0,\dots,k\}$, then the image $(z',w') = \Phi(z,w)$ of $(z,w)$ under $\Phi$ will be such that $z' \in B_{q'}$ and $w' = f( (t'_n)_{n \in \Z} )$, where the tape $(t'_n)_{n \in \Z}$ is obtained from $(t_n)_{n \in \Z}$ by first replacing $t_0$ with $t'_0$, and then shifting by $\epsilon$, where $q'$, $t'_0$ and $\epsilon$ are defined by \eqref{qat}.  Iterating this, we obtain Proposition \ref{pq} for the given universal Turing machine, and hence for arbitrary Turing machines.

\begin{remark} The above construction reveals in fact that the trajectory $y_s, \Phi(y_s), \Phi^2(y_s), \dots$ will either enter $U_{t_{-n},\dots,t_n}$, or stay a fixed distance away from this set.  Thus one only needs to be able to measure points in $M$ to some fixed non-zero accuracy in order to determine whether a given Turing machine with a given input halts or not, and to inspect a finite number of symbols of the output.
\end{remark}


\begin{thebibliography}{10}

\bibitem{arora}
S. Arora, B. Barak, Complexity Theory: A Modern Approach. Cambridge University Press, 2009.

\bibitem{gcb}
D. S. Gra\c{c}a, M. L. Campagnolo, J.  Buescu, Robust Simulations of Turing Machines with Analytic Maps and Flows. In: Cooper S.B., Löwe B., Torenvliet L. (eds) New Computational Paradigms. CiE 2005. Lecture Notes in Computer Science, vol 3526. Springer, Berlin, Heidelberg

\bibitem{gromov}
M. Gromov, \emph{Pseudo holomorphic curves in symplectic manifolds}, Invent. Math. \textbf{82} (1985), no. 2, 307--347. 

\bibitem{gunther}
M. G\"unther, \emph{Isometric embeddings of Riemannian manifolds}, Proceedings of the International Congress of Mathematicians, Vol. I, II (Kyoto, 1990), 1137–1143, Math. Soc. Japan, Tokyo, 1991. 

\bibitem{kcg}
P. Koiran, M. Cosnard, M. Garzon, \emph{Computability with low-dimensional dynamical systems}, Theoret. Comput. Sci. \textbf{132} (1994), no. 1-2, 113--128. 

\bibitem{lagarias}
J. Lagarias, \emph{An elementary problem equivalent to the Riemann hypothesis}, Amer. Math. Monthly \textbf{109} (2002), no. 6, 534--543.

\bibitem{nash}
J. Nash, \emph{The Imbedding Problem for Riemannian Manifolds}, Annals of Mathematics \textbf{63} (1956), 20--63.

\bibitem{sullivan}
D. Sullivan, \emph{A foliation of geodesics is characterized by having no "tangent homologies''},
J. Pure Appl. Algebra 13 (1978), no. 1, 101--104. 

\bibitem{tao-navier}
T. Tao, \emph{Finite time blowup for an averaged three-dimensional Navier-Stokes equation}, J. Amer. Math. Soc. \textbf{29} (2016), no. 3, 601--674.

\bibitem{tao-nlw}
T. Tao, \emph{Finite time blowup for a supercritical defocusing nonlinear wave system},  Anal. PDE \textbf{9} (2016), no. 8, 1999--2030.

\bibitem{overflow}
T. Tao, ``Finding a $1$-form adapted to a smooth flow'', 4 July 2017, {\tt mathoverflow.net/questions/273635}

\end{thebibliography}
\end{document}